\documentclass{article}
\usepackage{amsmath, amssymb, amsthm,authblk,bm,comment}
\usepackage{color}
\usepackage{graphicx} % Required for inserting images
\usepackage{hyperref}
\usepackage{pdfsync}

\newtheorem{theorem}{Theorem}
\newtheorem{proposition}{Proposition}

\newtheorem{remark}[theorem]{Remark}

\newtheorem{open}{Open problem}

\newcommand{\rr}{\mathbb R}

\newcommand{\gm}{\gamma}

\title{Variational Construction of Homoclinic and Heteroclinic  Orbits in the Planar Sitnikov Problem}
\author{Yuika Kajihara, Mitsuru Shibayama and Guowei Yu}

\begin{document}
\maketitle

\begin{abstract}
The Sitnikov problem is a special case of the three-body problem.
The system is known to be chaotic and has been studied by symbolic dynamics  (J. Moser, Stable and random motions in dynamical systems, Princeton University Press, 1973).
We study the limiting case of the Sitnikov problem as the eccentricity of the massive particles tends to 1.
By variational method,  we show the existence of  infinitely many homoclinic and heteroclinic solutions in the planar Sitnikov problem. In a previous work, for certain periodic symbolic sequences, the second author showed the existence of periodic solutions realizing them.
In this paper, we show the existence of homoclinic and heteroclinic solutions between some of these periodic orbits which realize certain non-periodic symbolic sequences.
\end{abstract}

\section{Introduction and Main Theorem}
Consider the three-body problem that is governed by the following ordinary differential equations (ODEs):
\begin{align}
    \ddot{q}_{1}&= -m_{2}\frac{q_{1}-q_{2}}{|q_{1}-q_{2}|^{3}} -m_{3}\frac{q_{1}-q_{3}}{|q_{1}-q_{3}|^{3}} \notag\\
    \ddot{q}_{2}&= -m_{1}\frac{q_{2}-q_{1}}{|q_{2}-q_{1}|^{3}} -m_{3}\frac{q_{2}-q_{3}}{|q_{2}-q_{3}|^{3}} \label{eqn:3bp}\\
    \ddot{q}_{3}&= -m_{1}\frac{q_{3}-q_{1}}{|q_{3}-q_{1}|^{3}} -m_{2}\frac{q_{3}-q_{2}}{|q_{3}-q_{2}|^{3}}.\notag
\end{align}
Here, $q_{k} \in \mathbb{R}^{3}$  and $m_k \ge 0$ for $k = 1,2,3$ represent positions and masses, respectively.

Sitnikov \cite{Sitnikov} considered a special case of the three-body problem in which two masses are equal and the other is zero, as well as the massless particle moves on the $z$-axis and the other two masses move along elliptic curves on the $xy$-plane, which are periodic orbits of the Kepler problem (Figure \ref{fig:Sitnikov}).
\begin{figure}
    \centering
    \includegraphics[width=3.5in]{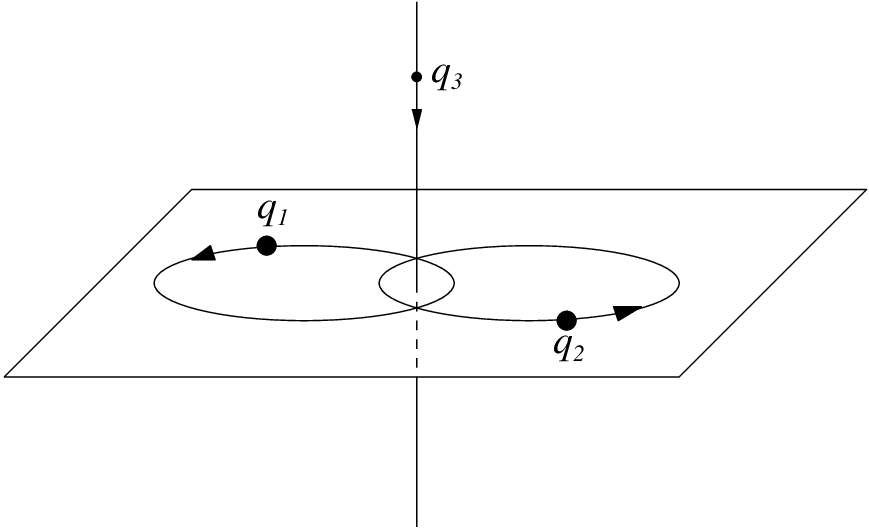}
    \caption{Sitnikov problem}
    \label{fig:Sitnikov}
\end{figure}
This became known as the Sitnikov problem, which is famous for the subsystem in which oscillatory orbits were first proven to exist. 
Using symbolic dynamics, the system was proven to be chaotic by Alekseev \cite{Alek} and Moser \cite{Moser}.

On the other hand, Chenciner \& Montgomery \cite{CM00} proved the existence of an eight-shaped periodic solution in the planar three-body problem with equal masses by variational method.
Since then, many periodic and relative periodic solutions have been found in the $N$-body problem as action minimizers in  different settings.

For the planar Sitnikov problem, in the previous work \cite{Shibayama},  the second author induced symbolic sequences and showed the existence of orbits realizing prescribed symbolic sequences with finite length and periodic orbits realizing prescribed periodic symbolic sequences by variational method. 
In this paper, we prove the existence of orbits realizing prescribed symbolic sequences with infinite length, which are homoclinic and heteroclinic orbits between periodic orbits, also using variational method.

Let $m_{1}=m_{2}=1$ and  $m_{3}=0$ for \eqref{eqn:3bp}, which can then be rewritten as 
\begin{align*}
    \ddot{q}_{1}&= -\frac{q_{1}-q_{2}}{|q_{1}-q_{2}|^{3}} \\
    \ddot{q}_{2}&= -\frac{q_{2}-q_{1}}{|q_{2}-q_{1}|^{3}}   & & (q_1, q_2, q_3 \in \mathbb{R}^2)\\
    \ddot{q}_{3}&= -\frac{q_{3}-q_{1}}{|q_{3}-q_{1}|^{3}} -\frac{q_{3}-q_{2}}{|q_{3}-q_{2}|^{3}}.
\end{align*}
The motion of  $q_{1}$ and $q_{2}$ is governed by the two-body problem.

There is a limiting case such that $q_{1}$ and $q_{2}$ move on the $x$-axis and $q_{3}$ moves on the $y$-axis.
We denote 
\[ q_{1}=(x(t), 0), q_{2}=(-x(t), 0), q_{3}=(0, y(t)).\] 
Here, $(x,y)$ is a solution that satisfies the following equations:
\begin{align*}
    \ddot{x}&= -\frac{1}{4  x^{2}} \\
    \ddot{y}&=-\frac{2y}{(x^{2}+y^{2})^{3/2}}.
\end{align*}
By rescaling, we can rewrite this as
\begin{align}
    \ddot{x}&= -\frac{1}{8  x^{2}}\label{eqn:x} \\
    \ddot{y}&=-\frac{y}{(x^{2}+y^{2})^{3/2}}.\label{eqn:y}
\end{align}
Let $x(t)$ be the periodic solution of \eqref{eqn:x} with period $1$ such that $x(t)=0 \ (t \in \mathbb{Z})$ and $x(t)>0 \ (t \in \mathbb{R} \backslash \mathbb{Z}$). 
The binary collision is regarded as regularized.
The regularization can be intuitively thought of as the elastic bounce (Figure \ref{fig:plsit}).
\begin{figure}[htbp]
    \centering
    \includegraphics[width=5cm]{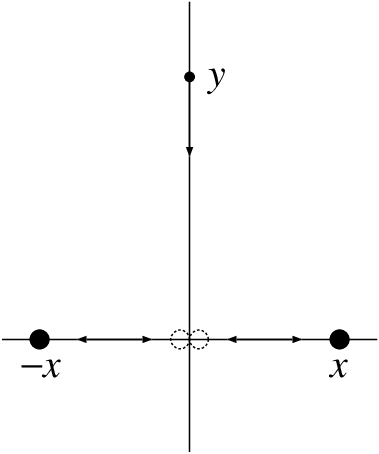}
    \caption{Planar Sitnikov problem}
    \label{fig:plsit}
\end{figure}

Consider a sequence $\bm{a}=\{a_{n}\}_{n \in \mathbb{Z}} \in \{-1, 1\}^{\mathbb{Z}}$.
Define the subset $\mathcal{M}$ of $\{-1, 1\}^\mathbb{Z}$ as the set of sequences $\bm{a}$ such that the length of  each successive part of $-1$ or $+1$ is no less than 3. 
More accurately, the sequence  $\bm{a} \in \{-1, 1\}^{\mathbb{Z}}$ is represented as 
\[\mathbf{a} : \dots, \underbrace{-1, -1, \dots, -1}_{k_{-1}}, \underbrace{1, 1, \dots, 1}_{{k_{0}}}, \underbrace{-1, -1, \dots, -1}_{k_{1}}, \underbrace{1, 1, \dots, 1}_{k_{2}}, \dots, \]
where $\{k_j\}_{j \in \mathbb{Z}}\subset \mathbb{N}$ is a bounded sequence. 
Then, let $\mathcal{M}$ be the set of $\bm{a}$ such that $k_j \ge 3$ for any $j \in \mathbb{Z}$.
Clearly, $\mathcal{M}$ is not empty for $N\ge 6$.

For $\bm{a} \in \mathcal{M}$ and $K_1,  K_2 \in \mathbb{Z}$ with $K_1<K_2$, we define three sets of functions $\Omega(\bm{a})$,  $\Omega^{\mathrm{per}}_{N}(\bm{a})$, and $\Omega_{K_1, K_2}(\bm{a})$ by
\begin{align*}
    \Omega(\bm{a}) &=\{ y(t) \in H^{1}_{\mathrm{loc}}(\mathbb{R}, \mathbb{R}) \mid a_{n} y(n) > 0 ~(n \in \mathbb{Z})\}, \\
    \Omega^{\mathrm{per}}_{N}(\bm{a})&=\{ y(t) \in H^{1}(\mathbb{R}/N\mathbb{Z}, \mathbb{R}) \mid a_{n} y(n) > 0 ~(n \in [0, N] \cap \mathbb{Z})\}, \ \text{and}\\
    \Omega_{K_1, K_2}(\bm{a}) &=\{ y(t) \in H^{1}([K_1, K_2], \mathbb{R}) \mid a_{n} \cdot y(n) > 0 ~(n  \in [K_1, K_2] \cap \mathbb{Z} )\}.
\end{align*} 
Moreover, we define subset $\mathcal{O}(\bm{a})_{N}$  of $\Omega(\bm{a})$ by
\begin{align*}
    \mathcal{O}(\bm{a})&= \{y \in  \Omega(\bm{a}) \mid \text{$y$ is a solution of the Sitnikov problem.}\}
\end{align*}
\begin{figure}[htbp]
\begin{center}
\includegraphics[height=5cm]{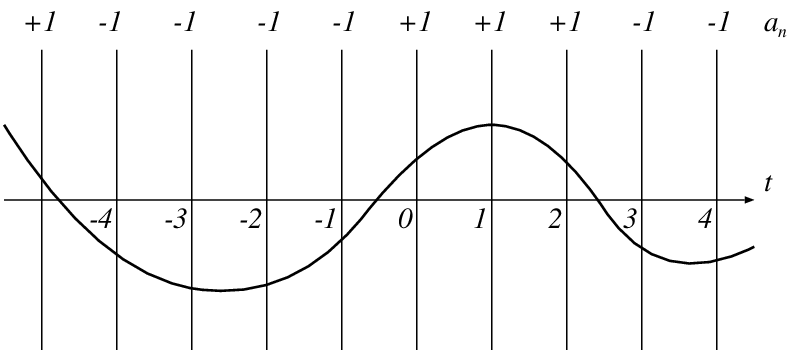}
\caption{Behavior of $y(t) \in \Omega(\bm{a})$}
\label{fig:curve}
\end{center}
\end{figure}
Figure \ref{fig:curve} shows an element of $\Omega(\bm{a})$ where $a_0=a_1=a_2=1$ and $a_{-1}=a_{-2}=a_{-3}=a_{-4}=-1$.

It is natural to assume that $\bm{a}$ is periodic when considering $\Omega^{\mathrm{per}}_{N}(\bm{a})$ and $\mathcal{O}^{\mathrm{per}}_{N}(\bm{a})$.
For $N \in \mathbb{N}$, let
\[
\mathcal{P}_N =\{\bm{a} \in \mathcal{M}\backslash \{\bm{e}^\pm\} 
\mid a_n = a_{n+N} \ ( n \in \mathbb{Z}) \}.
\]
where\[\bm{e}^+ = \dots, 1, 1, 1, \dots, \quad \bm{e}^- = \dots, -1, -1, -1, \dots.\]
Define a subset $\mathcal{O}^{\mathrm{per}}(\bm{a})$ of $\Omega(\bm{a})$ by
\begin{align*}
    \mathcal{O}^{\mathrm{per}}_{N} (\bm{a})&=\{y \in  \mathcal{O}(\bm{a}) \mid \text{$y$ is $N$-periodic.}\}
\end{align*}
The second author used minimizing methods to show $\mathcal{O}^{\mathrm{per}}_{N}(\bm{a}) \neq \emptyset$ for $N \ge 6$ \cite{Shibayama}.
More precisely, the planar Sitnikov problem has been shown to have infinitely many periodic orbits.
%%%%%%%%%%%%%%%%%%%%%%%%%%%%%%%%%%%%%%%%%%%%%%%%
\begin{theorem}[\cite{Shibayama}, Theorem 2] 
\label{theorem:Shibayama}
For any $\bm{a} \in \mathcal{P}_N$ with $N \ge 6$,
there exists an $N$-periodic solution $y \in \mathcal{O}^{\mathrm{per}}_{N}(\bm{a})$ satisfying the following conditions:
\begin{enumerate}
    \item For any $n \in \mathbb{Z}$ satisfying $a_n a_{n+1} =-1$, there exists a unique $t \in (n,n+1)$ such that $y(t)=0$, 
    \item For any $n \in \mathbb{Z}$ satisfying $a_n a_{n+1} =1$, it holds that $y(t)\neq 0$ for all $t \in (n,n+1)$, and
    \item The solution $y$ attains the infimum, i.e., 
    \[\mathcal{A}_{0,N}(y)=\inf_{x \in \Omega^{\mathrm{per}}_{N}(\bm{a})} \mathcal{A}_{0,N}(x),\]
    where
    \[\mathcal{A}_{s,t}(y) = \int_{s}^{t} L(y) dt
    =\int_{s}^{t}
    \frac{1}{2}\dot{y}^2+\frac{1}{\sqrt{(x(t))^2+y^2}}dt.
    \]
\end{enumerate}
\end{theorem}
%%%%%%%%%%%%%%%%%%%%%%%%%%%%%%%%%%%%%%%%%%%%%%%%

Roughly speaking, our main theorem (Theorem \ref{th:main}) implies that there exists a heteroclinic or homoclinic orbit between two periodic orbits obtained by Theorem \ref{theorem:Shibayama} under a symmetric assumption.
To state it, we define the subset $\mathcal{S}_N$ of $\mathcal{P}_N$ by 
\[ \mathcal{S}_N =\{\bm{a} \in \mathcal{P}_N \mid a_n=a_{-n} \ (n \in \mathbb{Z}) \},\]
which can also be written as
\[ \mathcal{S}_N =\{\bm{a} \in \mathcal{P}_N \mid a_j=a_{N-j}\ (j=0,1,\cdots,N) \}.\]
 The main result of this paper is the following theorem.
\begin{theorem} \label{th:main}
Suppose that $\bm{a}=\{a_n\} \in \mathcal{M}$
and $\bm{b}^\pm =\{b_n^\pm\}\in \mathcal{S}_{N^{\pm}}$
and there are $K^\pm (K^- <K^+)$ such that $a_{n}=b_n^+$ for $n \ge K^+ +1$ and $a_{n}=b_n^-$ for $n \le K^--1$.
Then, there are periodic solutions $\gamma^{\pm}$ in $\mathcal{O}^{\mathrm{per}}_{N^{\pm}} ( \bm{b}^\pm)$ and a connecting orbit $y(t)\in \mathcal{O}(\bm{a})$ from $\gamma^-$ to $\gamma^+$.
Moreover, for $n \in \mathbb{Z}$ satisfying $a_n a_{n+1}=1$, there is no real value $t \in (n, n+1)$ such that $y(t)=0$,
and for  $n \in \mathbb{Z}$  satisfying  $a_n a_{n+1}=-1$ there is a unique real number $t \in (n, n+1)$ such that $y(t)=0$.
\end{theorem} 

\begin{remark}
    Notice that $\bm{b}^+$ and $\bm{b}^-$ could be the same and then so are $\gm^+$ and $\gm^-$. In this case the connecting orbit is homoclinic, otherwise it is heteroclinic.
\end{remark}

This paper is simply structured, with Section \ref{section:gap} introducing a set of trajectories equipped with an order and a functional expressed as an infinite sum, following the approach used by Rabinowitz \cite{Rabinowitz04}.  
We examine its key properties and establish foundational results.  
The main theorem is proven in Section \ref{seq:mainproof}.  
The proof proceeds by considering minimizers of the functional introduced earlier, taken over an appropriate functional space.

%%%%%%%%%%%%%%%%%%%%%%%%%%%%%%%%%%%%%%%%%%%%%%%%%%
\section{Preliminary}
\label{section:gap}
The technical details differ, but the overall structure of the proof is similar to that presented in \cite{Rabinowitz04}.

\subsection{Setting}
Fix $\bm{b} \in \mathcal{P}_N$. 
Set
\[ \rho(\bm{b})= \inf \{\mathcal{A}_{0, N}(y) \mid y \in \Omega^{\mathrm{per}}_{N}(\bm{b})\} \]
and 
\[ \mathcal{N}(\bm{b})= \{ y \in \Omega_{N}^{\mathrm{per}} (\bm{b}) \mid \mathcal{A}(y) = \rho(\bm{b})\}.\]
%%%%%%%%%%%%%%%%%%%%%
Theorem \ref{theorem:Shibayama} implies that $\mathcal{N}(\bm{b})$ is non-empty. 
First, we show that the set of minimizers $ \mathcal{N}(\bm{b})$  has an order.
%%%%%%%%%%%%%%%%%%%%%
\begin{proposition}\label{prop:ordered}
$\mathcal{N}(\bm{b})$ is a totally ordered set, i.e.,  
$y_1, y_2 \in \mathcal{N}(\bm{b})$ implies 
that one of the following is true:
\begin{enumerate}
    \item $y_1(t)=y_2(t)$ for all $t \in \mathbb{R}$
    \item $y_1(t)>y_2(t)$ for all $t \in \mathbb{R}$
    \item $y_1(t)<y_2(t)$ for all $t \in \mathbb{R}$
\end{enumerate}
\end{proposition}
\begin{proof}
If not, there are points $t_1, t_2$ such that $y_1(t_1)=y_2(t_1)$ and $y_1(t_2)<y_2(t_2)$ (or $y_1(t_2)>y_2(t_2)$).
Set
\[
\phi(t)= \min \{y_1(t), y_2(t)\},
\ \text{and} \ \psi(t)= \max \{y_1(t), y_2(t)\}.
\]
Then, both $\phi(t)$ and $\psi(t)$ are also elements of $\mathcal{N}(\bm{b})$ because
\[2\rho(\bm{b}) \le \mathcal{A}_{0, N}(\phi) + \mathcal{A}_{0, N}(\psi)=  \mathcal{A}_{0, N}(y_1)+ \mathcal{A}_{0, N}(y_2)= 2 \rho(\bm{b}).\]
Hence, $\phi$ and $\psi$ are also solutions. 

In addition, $\phi(t)$ and $y_1(t)$ are identical on an interval including $t_2$.
By the uniqueness of solutions, $\phi(t)$ and $y_1(t)$ must be identical for all $t \in \rr$, and the same holds for $\psi(t)$ and $y_2(t)$. 
Therefore,
$y_1(t) \le y_2(t)$ for all $t \in \rr$, $y_1(t_1)=y_2(t_1)$, and $\dot{y}_1(t_1)=\dot{y}_2(t_1)$. 
Applying the uniqueness of the solutions again, we conclude that $y_1=y_2$.
\end{proof}
%%%%%%%%%%%%%%%%%%%%%%%%%%%%%%%%%%%%%%%%%%
\begin{proposition}
    There exist a maximum solution $\gamma_{\max}$ and a minimum solution $\gamma_{\min}$ in the ordered set $\mathcal{N}(\bm{b})$, such that for any $\gamma \in \mathcal{N}(\bm{b})$, the inequality
\[
\gamma_{\min}(t) \le \gamma(t) \le \gamma_{\max}(t)
\]
holds for all $t \in \mathbb{R}$.

\end{proposition}
\begin{proof}
Let 
\[ c=\inf \{ u(0) \mid u \in \mathcal{N}(\bm{b}) \}. \]
We first show that $c > -\infty$.
If $u \in \mathcal{N}(\bm{b})$, it holds that
\[ \int_0^N \frac{1}{2} \dot{u}^2 dt \le \mathcal{A}_{0, N}(u)= \rho (\bm{b}). 
\]
Since $\bm{b} \in \mathcal{P}_N$, 
there is a real number $t_0$ such that $u(t_0)=0$.
Let 
\[ u_{\min} = \min_{t \in [0, T]} u(t)=u(t_{\min}).\]
Thus, the Cauchy–Schwarz inequality implies
\[|u_{\min}| = \left| \int_{t_0}^{t_{\min}} \dot{u} dt\right| 
\le  \int_{0}^{N} |\dot{u} |dt
=\langle 1,  |\dot{u}| \rangle 
\le 
\sqrt{N} \left(\int_0^N \dot{u}^2 dt\right)^{1/2}.
\]
Therefore, we obtain
\[ |u_{\min}| \le \sqrt{2N\rho(\bm{b})}.  \]

Then, there exists a sequence $u_k \in \mathcal{N}(\bm{b}) $ such that 
$u_k(0) \to c$. 
From the lower semi-continuity of the action functional, 
there exists the limit $\gamma \in  \mathcal{N}(\bm{b})$ of $u_k$ 
that satisfies $\gamma(0)=c$.
Set
\[
\pi_0(\mathcal{N}(\bm{b}))
=\{y(0) \in \rr \mid y \in\mathcal{N}(\bm{b}) \}.
\]
We show that $\pi_0(\mathcal{N}(\bm{b}))$ is a closed set as follows.
Take $\{y_n\} \subset \mathcal{N}(\bm{b})$ such that $y_n(0) \to y(0)$.
The lower semi-continuity for $J$ yields
\[
\rho({\bm b}) =  \mathcal{A} \left( y_{n} \right) 
\ge \mathcal{A}(y),
\]
which implies $y \in \mathcal{N}(\bm{b})$ and $y(0) \in \pi_0(\mathcal{N}(\bm{b}))$.
\end{proof}
%%%%%%%%%%%%%%%%%%%%%%%%%%%%%%%%%%%%%%%%%%
For $\bm{b} \in \mathcal{P}_N$ and $k \in \mathbb{N}$, $k\bm{b} \in \mathcal{P}_{kN}$ is the $k$ number of successive $\bm{b}$s.
The set $\mathcal{N}$ is invariant under integer multiples of the sequence $\bm{b}$, as described in the following proposition.
%%%%%%%%%%%%%%%%%%%%%%%%%%%%%%%%%%%%%%%%%%
\begin{proposition}\label{prop:kper}
For any $k \in \mathbb{N}$,
$\mathcal{N}(k\bm{b})=\mathcal{N}(\bm{b})$ and $\rho(k\bm{b})=k \rho(\bm{b})$.
\end{proposition}
\begin{proof}
Let \( v \in \mathcal{N}(k\bm{b}) \).  
We show that \( v(t+N) = v(t) \) for all $t \in \rr$.  
Assume, for the sake of contradiction, that \( v(t_0+N) \neq v(t_0) \) for some $t_0 \in \mathbb{R}$.  
Since \( \mathcal{N}(k\bm{b}) \) is an ordered set, one of the following holds true:
\begin{enumerate}
    \item $v(t+N) > v(t) $ for all $ t \in \rr.$
    \item $v(t+N) < v(t) $ for all $ t \in \rr.$
\end{enumerate}  
Consider the case where $v(t+1) > v(t)$.  
Then, by periodicity,  
\[v(t) = v(t+kN) > v(t+(k-1)N) > \dots > v(t+N) > v(t),\]  
which is a contradiction.  
Similarly, if $v(t_0+N) < v(t_0)$, we would derive a similar contradiction using the periodicity of $v$.   
\end{proof}
%%%%%%%%%%%%%%%%%%%%%%%%%%%%%%%%%%%%%%%%%%
Moreover, any element $v \in \mathcal{N}(\bm{b})$ is also a minimizer under any two-point boundary conditions.
%%%%%%%%%%%%%%%%%%%%%%%%%%%%%%%%%%%%%%%%%%
\begin{proposition}
Let $v \in \mathcal{N}(\bm{b})$ and $a, b\in \mathbb{R}$ with $ a<b$.
Set 
\[
W= \{ w \in H^1 ([a, b]) \mid w(a)=v(a), \ w(b)=v(b), \ b_i w(i)>0 \ \text{for all} \ i \in [a,b] \cap \mathbb{Z}\}.
\]
Then $v$ is a minimizer of $\mathcal{A}_{a,b}$ on $W$, i.e.,  
\[ \mathcal{A}_{a,b}(v) = \inf_{w \in W}\mathcal{A}_{a,b}(w).\]
\end{proposition}
\begin{proof}
Assume that there is a curve $u \in W$ such that 
\[ \mathcal{A}_{a,b}(u) < \mathcal{A}_{a,b}(v).\]
Take integers $n^+$, $n^-$ and $k$ such that $n^- \le a< b \le n^+$
and $k N \ge n^+-n^-$.
Define $\phi$ by
\[ \phi(t) =\begin{cases} 
u(t) & ( t \in [a, b]) \\
v(t) & (t \in [n^-, a) \cup (b, n^- + kN] ).
\end{cases}
\]
The modified curve $\phi$ belongs to $\mathcal{N}(k\bm{b})$ and 
has a lower value of the action functional than $v$, which contradicts Proposition \ref{prop:kper}.
\end{proof}
%%%%%%%%%%%%%%%%%%%%%%%%%%%%%%%%%%%%%%%%%%
The above properties have been established under the assumption that  $\bm{b} \in \mathcal{P}_N$.
Restricting this to $\bm{b} \in \mathcal{S}_N$, $\mathcal{N}(\bm{b})$ is a set of symmetric trajectories in the following sense.
%%%%%%%%%%%%%%%%%%%%%%%%%%%%%%%%%%%%%%%%%%
\begin{proposition}
For $\bm{b} \in \mathcal{S}_N$ and $v \in \mathcal{N}(\bm{b})$, $v(t)=v(-t)$ for all $t \in \rr$.
\end{proposition}
\begin{proof}
Set a function $w$ by $w(t)=v(-t)$.
Clearly $w(0)=v(0)$ and $w \in \mathcal{N}(\bm{b})$.
By Proposition \ref{prop:ordered}, $v(t)=v(-t)$ for all $t$.
\end{proof}
%%%%%%%%%%%%%%%%%%%%%%%%%%%%%%%%%%%%%%%%%
Symmetries imply that global minimizers are periodic, as follows.
\begin{proposition}\label{prop:ygerho}
Let $\bm{b} \in \mathcal{S}_N$.
For any $y \in \Omega_{0, N}(\bm{b})$, 
$\mathcal{A}_{0, N}(y) \ge \rho(\bm{b})$. 

\end{proposition}
\begin{proof}
Take any $y \in \Omega_{0,N}(\bm{b})$.
Let
\[
\phi(t)=\max\{y(t), y(N-t)\}, \ \text{and} \ \psi(t)=\min\{y(t), y(N-t)\}.
\]
Then, $\phi, \psi \in \Omega_N^{\mathrm{per}} (\bm{b})$. 
Therefore,
\[2 \rho(\bm{b}) \le  \mathcal{A}_{0, N}(\phi)+\mathcal{A}_{0, N}(\psi) 
= 2\mathcal{A}_{0, N}(y),\]
which is the desired result.
\end{proof}
%%%%%%%%%%%%%%%%%%%%%%%%%%%%%%%%%%%%%%%%%%
\subsection{Normalization}
Next, we define a normalized functional on $H^1_{\mathrm{loc}}(\mathbb{R})$.
For $u \in H^1_{\mathrm{loc}}(\mathbb{R})$,  $\bm{b^+} \in \mathcal{S}_{N^+}$,
and $\bm{b^-} \in \mathcal{S}_{N^-}$,
set 
\[ J(u;\bm{b}^-,\bm{b}^+)= \sum_{p \in \mathbb{Z}} a_p(u)\]
where
\begin{align*}
    a_p(u) =
    \begin{cases}
        \mathcal{A}_{p,p+1}- \frac{\rho(\bm{b^-})}{N^-} & (p < 0)  \\
        \mathcal{A}_{p,p+1}- \frac{\rho(\bm{b^+})}{N^+} & (p \ge 0) .
    \end{cases}
\end{align*}
We abbreviate $J(u;\bm{b}^-,\bm{b}^+)$ as $J(u)$.
In the following, we always assume the condition of Theorem \ref{th:main}, that is, 
there are $K^+$ and $K^-$ with $ K^- <K^+$ such that $a_{n}=b_n^+$ for $n \ge K^{+} +1$ and $a_{n}=b_n^-$ for $n \le K^{-}-1$.
We now discuss the basic properties of $J$.
%%%%%%%%%%%%%%%%%%%%%%%%%%%%%%%%%%%%%%%%%%
\begin{proposition}
The infimum of $J$ is bounded below, i.e.,
\[\inf_{u \in \Omega(\bm{a})} J(u) > -\infty.\]
\end{proposition}
\begin{proof}
Let $K^+$ and $K^-$ be 
\begin{align}
\label{eq;kpm}
\begin{split}
K^+ &= \max\{n \in \mathbb{Z} \mid a_n \neq b_n^+ \},  \ \text{and} \\
K^- &= \min\{n \in \mathbb{Z} \mid a_n \neq b_n^- \}. 
\end{split}
\end{align} 
Assume $l$ satisfies $lN^{+}+1 \ge K^+$. 
Proposition \ref{prop:ygerho} implies
\[ \sum_{p=lN+1}^{(l+1)N} a_p (u) \ge 0. \]
The case of $(l+1)N^- \le K^-$ is similar.
Since $\sum_{p=K^-}^{K^+} a_p(u)$ is bounded from below, so is $J$. 
\end{proof}

\begin{proposition}
There exists a curve $v \in \Omega(\bm{a})$ 
such that 
\[ J(v) < \infty.\]
\end{proposition}
\begin{proof}
    Take $\xi^\pm$ as a minimizer of the action functoinal $\mathcal{A}_{0, N^{\pm}}$ in $\Omega_{N^\pm}^{\mathrm{per}}$ guaranteed by Theorem \ref{theorem:Shibayama}. 
Let $\xi^0$ be an arbitrary smooth path in $\Omega_{K^+ K^-}(\bm{a})$
such that $\xi^{0
}(K^\pm)=\xi^{\pm}(K^\pm)$. 
Define 
\[ \xi(t) = \begin{cases} 
\xi^-(t) & (t <K^-) \\
\xi^0(t) & (K^- \le t \le K^+ )\\
\xi^+(t) & (t > K^+). 
\end{cases} 
\]
For $l N^++1 \ge K^+$, 
\[ \sum_{p=lN^++1}^{(l+1)N^+} a_p (\xi)=0.  \]
For $l N^- \le K^-$, 
\[ \sum_{p=(l-1)N^-+1}^{l N^-} a_p (\xi)=0.  \]
Since \[ \sum_{p =lN^- +1}^{l N^+} a_p (\xi) \]
is bounded, 
so is $J(\xi)$. 
\end{proof}

\section{Proof of Theorem \ref{th:main}}
\label{seq:mainproof}

We will show the existence of the heteroclinic orbit
$y^\ast \in \Omega(\bm{a})$ from $\gamma^-$ to $\gamma^+$ and focus on the proof in the case of $a_{K^+}=1$ and $b_{K^+}^+=-1$.
Other cases, including homoclinic orbits, can be shown in a similar way.
Notice that if 
$a_{K^+}=1$ and $b_{K^+}^+=-1$, 
we can let 
$\gamma^+$ be the maximum solutions in $\mathcal{N}(\bm{b}^+)$
in the meaning  of Proposition \ref{prop:ordered}.
If 
$a_{K^+}=-1$ and $b_{K^+}^+=1$, let 
$\gamma^+$ be the mininum solutions in $\mathcal{N}(\bm{b}^+)$.
Similarly, 
if $a_{K^-}=1$ and $b_{K^-}^-=-1$, 
let 
$\gamma^-$ be the maximum solutions in $\mathcal{N}(\bm{b}^-)$.
If $a_{K^-}=-1$ and $b_{K^-}^-=1$,  let 
$\gamma^-$ be the mininimum solutions in $\mathcal{N}(\bm{b}^-)$.

Before the proof, we set new functional spaces.
We define
\begin{align*}
\Gamma_{-\infty} &= \{ u \in H^1_{\mathrm{loc}}(\mathbb{R}, \mathbb{R}) \mid 
\| u - \gamma^- \|_{L_2([i, i+1])} \to 0 \ ( i \to - \infty) \}, \ \text{and}\\
\Gamma_\infty &= \{ u \in H^1_{\mathrm{loc}}(\mathbb{R}, \mathbb{R}) \mid 
\| u - \gamma^+ \|_{L_2([i, i+1])} \to 0 \ ( i \to \infty) \} .
\end{align*}
In addition, set
\[ \Gamma_1(\bm{a}) = \Omega(\bm{a}) \cap \Gamma_{-\infty} \cap \Gamma_{\infty}\]
and 
\[ c_1 (\bm{a})= \inf \{ J(y) \mid y \in \Gamma_1(\bm{a}) \}.\]

Now we are ready to prove our main theorem.
\begin{proof}[Proof of Theorem \ref{th:main}]
It suffices to show that
\[ \mathcal{M}(\bm{a}) = \{y \in \Gamma_1 (\bm{a}) \mid J(y)=c_1(\bm{a}) \} \neq \emptyset.\]
Let $u_k$ be a minimizing sequence of $J$ on $\mathcal{M}(\bm{a})$.
We can take a subsequence $z_k$ of $u_k$ that converges to some $y^\ast$ in $H^1([-1, 1])$. 
Since $y^\ast$ is a minimizer of $\mathcal{A}_{-1, 1}$ under 
a fixed-ends constraint,
it is also a solution satisfying 
\[ a_k y^\ast (k) >0 \qquad (k=-1, 0, 1).\]
Take an arbitrary non-negative integer $l$. 
We have a subsequence $w_k^{(l)}$ of $z_k$ that converges to some $\alpha^{(l)}$ in $H^1([-l, l])$. 
The minimizers have no total collision \cite{Shibayama}.
Hence, $\alpha^{(l)}$ is a solution satisfying 
\[ a_k \alpha^{(l)} (k) >0 \qquad (k=-l, 0, l).\]
Since $w_k^{(l)}$ is a subsequence of $z_k$, 
we see that $w_k^{(l)}$ also converges to $y^\ast(t)$ in $H^1([-1, 1])$.
Thus, $w_k^{(l)}$ is identical with $y^\ast(t)$. 
Since $l$ is an arbitrary non-negative integer,
$y^\ast(t)$ satisfies 
\[ a_k y^\ast (k) >0 \qquad (k \in \mathbb{Z}).\]

Let $y^\ast \in \mathcal{M}(\bm{a})$. 
We show $y^\ast < \gamma^+$ on $[K^+, \infty)$.
Assume that this inequality does not hold. 
Since $y^\ast(K^+)<0 < \gamma^\ast(K^+)$, 
$y^\ast$ and $\gamma^+$ intersect on $(K^+, \infty)$. 
Let $y^\ast(t_0)=\gamma^+(t_0)$ and $t_0 \in (K^+, \infty)$.
Define 
\[ z(t) =\begin{cases}
y^\ast (t) & (t \le t_0) \\
\gamma^+(t) & (t > t_0). 
\end{cases}\]
From Proposition \ref{prop:ygerho}, $J(z) \le J(y^\ast)$. 
Therefore, $z$ is also a minimizer of $J$ and ,hence,  $z$ is a solution. 
Since $z$ and $\gamma^+$ are identical for all $t \ge t_0$ and 
not identical for some $t < t_0$, this contradicts the uniqueness of 
solution.

Since $J(y^\ast) < \infty$, 
\[ \lim_{l \to \pm \infty} \sum_{p=lN^++1}^{(l+1)N^+}   a_p (y^\ast) = 0. \] Therefore, a subsequence
\[u_l(t)= y^\ast( t +(lN^++1)) \qquad (t \in [0, N^+]). \]
converges to one of $\mathcal{N}(\bm{b}^+)$. 
Finally, we show that $u_l$ tends to $\gamma^+$.
Assume that $u_l$ does not converge to $\gamma^+$.
Then, there is a subsequence $v_k$ of $u_l$ that is uniformly $\delta>0$ away from $\gamma^+$.
Since $v_k$ is a minimizing sequence, there is a subsequence of $v_k$ that converges
to a minimizer, which should be $\gamma^+$.
This is a contradiction. 
Hence, $u_l$ converges to $\gamma^+$ as $l \to \infty$.
Similarly $u_l$ converges to $\gamma^-$ as $l \to -\infty$.
Consequently $u_l$ is a heteroclinic solution connecting $\gamma_-$ to $\gamma^+$.
\end{proof}

To conclude the paper, we pose an open problem.
\begin{open}
    Are there similar heteroclinic and homoclinic orbits of our theorems in the case of the spatial Sitnikov problem as seen in Figure \ref{fig:Sitnikov}?
\end{open}
At first glance, this problem seems simple as we do not need to show collisionlessness.  
However, it turns out to be difficult to prove because, in our current setting, binary collisions play an essential role in determining symbolic sequences and gap pairs.  
In the elliptic case, it may not be possible to determine whether all three bodies lie on a straight line at the moment they are on the \( y = 0 \) plane.  
Therefore, it is difficult to apply the same approach as in the present study.

\end{document}